\documentclass[11pt,a4paper]{article}

\usepackage{lmodern}
\usepackage{graphicx,color}
\usepackage{amsmath,amsfonts,amsthm,eucal}
\usepackage{bm,bbm}
\usepackage{hyperref,url,cite}

\allowdisplaybreaks[3]

\renewcommand{\Tilde}{\widetilde}

\newcommand{\NN}{\mathbb{N}}     
\newcommand{\ZZ}{\mathbb{Z}}     
\newcommand{\RR}{\mathbb{R}}     
\newcommand{\CC}{\mathbb{C}}     

\newcommand{\cS}{\mathcal{S}}   
\newcommand{\cD}{\mathcal{D}}

\newcommand{\dt}{\frac{3}{2}}

\renewcommand{\epsilon}{\varepsilon}

\newcommand{\dd}{\,\mathrm{d}}
\newcommand{\ess}{\mathrm{ess}}

\DeclareMathOperator{\ran}{ran}
\DeclareMathOperator{\dom}{dom}
\DeclareMathOperator{\spec}{spec}

\newtheorem{theorem}{Theorem}
\newtheorem{prop}[theorem]{Proposition}
\newtheorem{lemma}[theorem]{Lemma}
\newtheorem{corol}[theorem]{Corollary}

\theoremstyle{definition}

\newtheorem{defin}[theorem]{Definition}

\newtheorem{remark}[theorem]{Remark}

\begin{document}

\title{\bf On Neumann-Poincar\'e operators and self-adjoint transmission problems}

\author{\large Badreddine Benhellal\,\dag{} \large and Konstantin Pankrashkin\,\ddag\\[\medskipamount]
	\small Carl von Ossietzky Universit\"at Oldenburg\\
	\small Fakult\"at V, Institut f\"ur Mathematik\\
	\small 26111 Oldenburg, Germany\\[\medskipamount]
	\small \dag{} E-Mail: \url{badreddine.benhellal@uol.de}\\[\medskipamount]
	\small \ddag{} ORCID: 0000-0003-1700-7295 \\
	\small E-Mail: \url{konstantin.pankrashkin@uol.de},\\
	\small Webpage: \url{http://uol.de/pankrashkin}
}

\date{}

\maketitle

%
%
%
%

\maketitle	

\begin{abstract}
We discuss the self-adjointness in $L^2$-setting of the operators acting as $-\nabla\cdot h\nabla$,
with piecewise constant functions $h$ having a jump along a Lipschitz hypersurface $\Sigma$, without
explicit assumptions on the sign of $h$. We establish a number of sufficient conditions for the self-adjointness
of the operator with $H^s$-regularity for suitable $s\in[1,\frac{3}{2}]$, in terms of the jump value and the regularity
and geometric properties of $\Sigma$. An important intermediate step is a link
with Fredholm properties of the Neumann-Poincar\'e operator on $\Sigma$, which is new
for the Lipschitz setting.
\end{abstract}

\thispagestyle{plain}
\pagestyle{plain}

\section{Introduction}

Let $n\ge 2$ and $\Omega\subset\RR^n$ be a non-empty bounded open set with Lipschitz boundary. 
Let $\Omega_-\subset\Omega$ be a non-empty open subset with Lipschitz boundary such that $\overline{\Omega_-}\subset\Omega$, and set 
\begin{align}
	\Omega_+=\Omega\setminus \overline{\Omega_-}, \quad \quad\Sigma=\partial\Omega_-,
\end{align}
see Fig.~\ref{fig1}. In order to avoid combinatorially involved configurations we assume from the very beginning that
\[
\Sigma \text{ is connected}.
\]

\begin{figure}
	\centering
	\includegraphics[height=35mm]{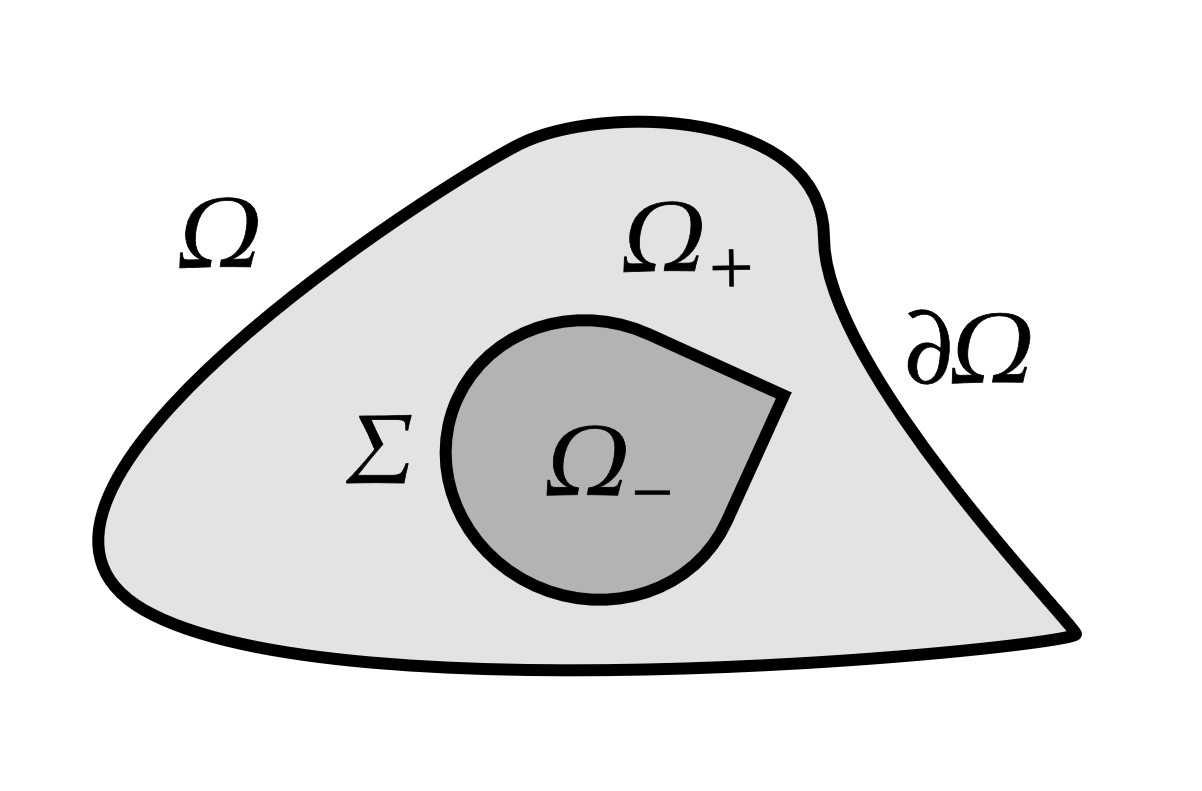}
	\caption{The sets $\Omega_\pm$ and $\Sigma$.}\label{fig1}
\end{figure}

In the present work, we will be interested in the self-adjointness of the operator formally acting as $u\mapsto -\nabla\cdot h\nabla u$
in $L^2(\Omega)$ with the Dirichlet boundary conditions at the exterior boundary $\partial\Omega$, where $h:\Omega\to \RR$ is the piecewise constant function
\[
h:\ \Omega\ni x\mapsto \begin{cases}
	1, & x\in \Omega_+,\\
	\mu, & x\in \Omega_-
\end{cases}
\]
and $\mu\in\RR\setminus\{0\}$ is a parameter. For $\mu>0$ the operator can be defined using Lax-Milgram representation theorem and the ellipticity theory,
and it is semibounded from below with compact resolvent, so we are mainly interested in the case $\mu<0$. Various equations with the above operators for $\mu<0$ appear in the mathematical theory of negative-index metamaterials, stemming from the pioneering works by Veselago~\cite{ves}, and they
show a number of unexpected properties compared to the usual elliptic case~\cite{bouch,kohn}. We also remark that there is an increasing interest to sign-changing problems in the homogenization theory and the uncertainty quantification \cite{waurick,nistor1,nistor2}.

In what follows, any function $u\in L^2(\Omega)$ will be identified with the pair $(u_+,u_-)$ with $u_\pm$ being its restrictions on $\Omega_\pm$. This gives rise to the idenfitication $L^2(\Omega)\simeq L^2(\Omega_+)\oplus L^2(\Omega_-)$. Then the above operator is expected to act,
at least at a naive level, as $(u_+,u_-)\mapsto (-\Delta u_+,-\mu \Delta u_-)$
on the functions $u_\pm$ should satisfying the boundary and transmission conditions
\begin{equation}
	\label{trans0}
u_+=0 \text{ on } \partial\Omega, 
\quad u_+=u_- \text{ on } \Sigma,
\quad \partial_+ u_++\mu\partial_- u_- \text{ on } \Sigma,
\end{equation}
with $\partial_\pm$ being the outward normal (with respect to $\Omega_\pm$) derivative on $\Sigma$. 
Finding precise regularity of the functions $u_\pm$ guaranteeing the self-adjointness
of the resulting operator is a non-trivial task depending on a combination of the parameter $\mu$ (usually called contrast) and the geometric properties of $\Sigma$, and this is the central topic of the present paper.

It seems that the problem was first addressed in \cite{bbdr} for the case when
both $\Omega_\pm$ are  $C^2$-smooth. Using the results of \cite{cost} it turned
out that the requirement $u_\pm\in H^1(\Omega_\pm)$ with $\Delta u_\pm\in L^2(\Omega_\pm)$ leads to a self-adjoint operator for all $\mu\ne -1$. On the other hand, if $n=2$ and $\Sigma$ has corners, then the operator is only self-adjoint if $\mu$ lies outside some non-trivial critical interval, see Remark \ref{rmk21} below. A similar picture was found for some boundary singularities in higher-dimensional situations as well, see \cite{bccj,bcc,acr}.

It should be noted that the regularity of functions in the operator domain turns out to be crucially important. For example, in the papers \cite{jd,cpp,kroum} it was found that for some values of $\mu$ the operator with $u_\pm\in H^2(\Omega_\pm)$ is not closed; its closure is a self-adjoint operator whose essential spectrum can be non-empty, which in parts explains various
regularization issues found in earlier works \cite{bouch,kohn,kett}. We remark that very
similar effects
were recently discovered in transmission problems for Dirac operators as well~\cite{BHSS,BHOP,BP}.

It seems that no previous work addressed the case of general Lipschitz $\Sigma$ so far, and we are trying to close this gap in the present paper. 
One should mention the existence of an approach based on modified representation theorems
for quadratic forms \cite{mat,mfo,ss1,ss2}, which allows for the study of variable coefficients, 
but the final results on the self-adjointness require higher regularity of the hypersurface $\Sigma$ or suitable symmetries of the sets $\Omega_\pm$.

For each $s\in[1,\frac{3}{2}]$ denote by $A_{(s)}$ the linear operator in $L^2(\Omega)$ acting as $(u_+,u_-)\mapsto (-\Delta u_+,-\mu\Delta u_-)$
on the functions $u_\pm\in H^s(\Omega_\pm)$ such that $\Delta u_\pm\in L^2(\Omega_\pm)$ and that
the boundary and transmission conditions \eqref{trans0} are satisfied (a precise definition
of the boundary values will be given in Sections \ref{secint} and \ref{sec3}).
Our main result is a link between the self-adjointness of $A_{(s)}$ and the spectral properties
of the so-called Neumann-Poincar\'e operator $K'_\Sigma$ on $\Sigma$, which is the singular integral operator defined by
\[
	K'_\Sigma f (x)=\mathrm{p.v.}\int_\Sigma \frac{\big\langle \nu(x),x-y\big\rangle}{\sigma_n|x-y|^n} f(y)\dd s(y),
\]
where $\nu$ is the unit normal field of $\Sigma$ pointing to $\Omega_+$ and $\mathrm{p.v.}$
stands for the Cauchy principal value. The Neumann-Poincar\'e operator appears naturally when considering the boundary values of layer potentials, see Section~\ref{secint}. Its spectral analysis attracted a lot of attention during the last decades \cite{ando,hag,km2,kmp1,Miy,MR}, and it already appeared in various problems related to the existence of solutions in sign-indefinite problems \cite{grieser,ng1,ng2,ola}. However, it seems that the links between the self-adjointness of $A_{(s)}$ and the spectrum of $K'_\Sigma$ were never addressed explicitly for the case of general Lipschitz $\Sigma$. Our main observation is as follows
(Corollary \ref{corself2}):
\[
\parbox{110mm}{
Let $\mu\in\RR\setminus\{-1,0\}$ and $s\in\{1,\frac{3}{2}\}$ be such that  $K'_\Sigma-\frac{\mu+1}{2(\mu-1)}$ viewed as an operator $H^{s-\frac{3}{2}}(\Sigma)\to H^{s-\frac{3}{2}}(\Sigma)$ is Fredholm with index zero, then $A_{(s)}$ is self-adjoint with compact resolvent.
}
\]
We remark that various results for more general values of $s$ are shown throughout the text as well, but they require a more involved formulation, see e.g. Theorem \ref{corself}. Using the above result, we establish several specific situations in which $A_{(s)}$ is self-adjoint with compact resolvent:
\begin{itemize}
	\item the sign-definite case $\mu>0$ and $s=\frac{3}{2}$ (Theorem \ref{thm17}),
	\item $s=\frac{3}{2}$ and for all $\mu\ne -1$, if $\nu$ has a vanishing mean oscillation, which covers all $C^1$-smooth $\Sigma$ (Theorem \ref{thm18}),
	\item $n=2$, $\Sigma$ is a curvilinear polygon with the sharpest corner $\omega\in(0,\pi)$,
	\begin{itemize} 
	\item[$\circ$] $s=\frac{3}{2}$,  if the edges of $\Sigma$ are $C^1$-smooth and
	\[
	\mu\notin \Big[-\cot^2\frac{\omega}{4},-\tan^2\frac{\omega}{4}\Big].
	\]
        \item[$\circ$] $s=1$, if the edges of $\Sigma$ are $C^2$-smooth and
	\[
	\mu\notin \Bigg[-\dfrac{\pi+|\pi-\omega|}{\pi-|\pi-\omega|}, -\dfrac{\pi-|\pi-\omega|}{\pi +|\pi-\omega|}\,\Bigg],
	\]
\end{itemize}
see Theorem \ref{thm20}.
\item $n=3$, $s=1$, $\Sigma$ is a surface of revolution which is $C^1$-smooth except at a single point at which
it has a conical singularity of opening angle $\alpha\in(0,\pi)\setminus\{\frac{\pi}{2}\}$, and
\begin{itemize}
	\item[$\circ$]  $\mu>-1$ for $\alpha\in(0,\frac{\pi}{2})$,
	\item[$\circ$]  $\mu<-1$ for $\alpha\in(\frac{\pi}{2},\pi)$,
\end{itemize}
see Theorem \ref{thmcon} for a detailed formulation.
\end{itemize}

As discussed in Remark \ref{rmk19} below, the required spectral bounds for $K'_\Sigma$ belong actually to the most important conjectures
in the theory of Neumann-Poincar\'e operators. Very few explicit results are available so far, which in part explains why the self-adjointness problem for the sign-indefinite case turned out to be unexpectedly difficult. We hope that the link between the two classes of problems we establish in this paper will draw the attention of new communities to the respective questions.

Our approach follows mainly the philosophy of quasi boundary triples for symmetric operators \cite{bl07,bl12},
which is enhanced by the very recent sharp trace theory on Lipschitz domains developed in \cite{bgm}.
Our main step is establishing a resolvent formula for the operator $A_{(s)}$ (Theorem \ref{propaa}).
Then the self-adjointness is reduced to proving the surjectivity of $A_{(s)}-z$ for suitable $z$,
which is established with the help of the mapping properties of various boundary integral operators borrowed mainly from~\cite{verchota}.

\section{Boundary integral operators and Dirichlet-to-Neumann maps in Lipschitz domains}\label{secint}

Let $n\ge 2$ and $U\subset\RR^n$ be a non-empty open set with Lipschitz boundary $\partial U$ and $\nu$ be the outer unit normal on $\partial U$.

We denote by $H^s$ the usual Sobolev spaces of order $s\in\RR$. In addition, we denote
\[
H^s_\Delta(U):=\big\{f\in H^s(U):\, \Delta f\in L^2(U)\big\},
\]
which will be equipped with the norm $\|f\|^2_{H^s_\Delta(U)}:= \|f\|^2_{H^s(U)}+\|\Delta f\|^2_{L^2(U)}$.

Let $\gamma^{\partial U}_D:H^1(U)\to L^2(\partial U)$ be the Dirichlet trace map, i.e. 
\[
\gamma^{\partial U}_D f:=f|_{\partial U} \text{ for $f\in C^\infty(\overline U)$,}
\]
which is extended by density and continuity. In fact, better mapping properties are known \cite[Corollary 3.7]{bgm}:

\begin{prop}\label{prop1} For any $s\in[\frac{1}{2}, \frac{3}{2}]$
the map $\gamma^{\partial U}_D:H^s_\Delta(U)\to H^{s-\frac{1}{2}}(\partial U)$ is bounded and surjective, and there exists
a bounded linear map
\[
\Upsilon:H^{s-\frac{1}{2}}(\partial U)\to H^s_\Delta(U)
\]
such that
\[
\Delta \Upsilon \varphi=0 \text{  and }
\gamma^{\partial U}_D \Upsilon \varphi=\varphi
\text{ for any }\varphi\in H^{s-\frac{1}{2}}(\partial U).
\]
\end{prop}

In addition, let $\gamma^{\partial U}_N:H^2(U)\to L^2(\partial U)$ be the Neumann trace, 
\[
\gamma^{\partial U}_N f:=\langle \nu,\nabla f|_{\partial U}\rangle \text{ for } f\in C^\infty(\overline U),
\]
extended by density and continuity. The following properties are known \cite[Corollary 5.7]{bgm}:
\begin{prop}\label{prop12}
For any $s\in[1,\frac{3}{2}]$ the map $\gamma^{\partial U}_N: H^{s}(U)\to H^{s-\frac{3}{2}}(\partial U)$ is well-defined, bounded and surjective, and
for any $f,g\in H^s_{\Delta}(U)$ one has the usual Green formula (integration by parts)
\begin{align*}
	\langle f,\Delta g\rangle_{L^2(U)}-\langle \Delta f,g\rangle_{L^2(U)}
&	= \langle \gamma^{\partial U}_D f,\gamma^{\partial U}_N g\rangle_{H^{\frac{3}{2}-s}(\partial U),H^{s-\frac{3}{2}}(\partial U)}\\
&\quad- \langle \gamma^{\partial U}_N f,\gamma^{\partial U}_D g\rangle_{H^{s-\frac{3}{2}}(\partial U),H^{\frac{3}{2}-s}(\partial U)}.
\end{align*}
\end{prop}

Denote by $\Phi$ the fundamental solution of the Laplace equation in $\RR^n$ defined by
\begin{equation*}
	\Phi:\quad \RR^n\setminus\{0\}\ni x\mapsto \begin{cases}
		\dfrac{1}{2\pi}\log|x|,& \text{for }n = 2,\\[\smallskipamount]
		\dfrac{1}{\sigma_n (n-2)|x|^{n-2}}, &\text{for }n \geq 3,
		\end{cases}
\end{equation*}
where $\sigma_n$ stands for the hypersurface measure of the unit sphere in $\RR^n$.
Define the single layer boundary operator $S_{\partial U}:L^2(\partial U)\to L^2(\partial U)$ by
\[
S_{\partial U} f(\cdot ):=\int_{\partial U}\Phi(\cdot-y)f(y)\dd s(y).
\]
It is clear from the definition that $S_{\partial U}:L^2(\partial U)\to L^2(\partial U)$ is bounded and self-adjoint.
In addition, $\ran S_{\partial U}\subset H^1(\partial U)$ and
\[
S_{\partial U}:\ L^2(\partial U)\to H^1(\partial U)
\]
is a bounded operator as well, see \cite[Lemma 1.8]{verchota}.

We define the double layer operator $K_{\partial U}$ by
\[
	K_{\partial U} f (\cdot)=\mathrm{p.v.}\int_{\partial U} \frac{\langle \nu(y),y-\cdot\rangle}{\sigma_n|\cdot-y|^n} f(y)\dd s(y),
\]
and its formal adjoint $K'_{\partial U}$ by
\[
	K'_{\partial U} f (\cdot)=\mathrm{p.v.}\int_{\partial U} \frac{\langle \nu(\cdot ),\cdot -y\rangle}{\sigma_n|\cdot-y|^n} f(y)\dd s(y).
\]
The following result is available, see \cite[Theorem 7.1]{Mc} for $s\in(0,1)$ and \cite[Theorems 3.3 and 4.12]{verchota}
for $s\in\{0,1\}$:
\begin{prop} For any $s\in[0,1]$ the operators
\[
K_{\partial U}: H^s(\partial U)\to H^s(\partial U), \quad K'_{\partial U}: H^{-s}(\partial U)\to H^{-s}(\partial U)
\]
are bounded.	
\end{prop}

\begin{defin}
The Dirichlet Laplacian $-\Delta_{U}$ associated with $U$ is the linear operator in $L^2(U)$ defined by
\[
\dom (-\Delta_U):=\big\{
f\in H^\dt_\Delta(U):\ \gamma^{\partial U}_D f=0
\big\},
\quad
-\Delta_U:\ f\mapsto -\Delta f.
\]
It is known to be a self-adjoint operator with compact resolvent \cite[Theorem 6.9]{bgm}, and
for any $z\in\CC\setminus \spec (-\Delta_U)$ the resolvent $(-\Delta_U-z)^{-1}$
defines a bounded operator $L^2(U)\to H^\dt_\Delta(U)$.
\end{defin}

\begin{lemma}\label{lemma5} Let $s\in[0,1]$, $\varphi\in H^s(\partial U)$ and $z\in \CC\setminus \spec (-\Delta_U)$, then
	the boundary value problem
\begin{equation*}
(-\Delta-z) f=0 \text{ in } U,
\quad \gamma^{\partial U}_D f=\varphi, \quad f\in H^{s+\frac{1}{2}}(U),
\end{equation*}
has a unique solution, which is given by
\[
f:=z(-\Delta_U-z)^{-1}\Upsilon \varphi+\Upsilon\varphi
\]
with $\Upsilon$ defined in Proposition \ref{prop1}. The resulting map $\varphi\mapsto f$, i.e.
\begin{align}\label{PO}
\begin{split}
P^U_z&:=z(-\Delta_U-z)^{-1}\Upsilon +\Upsilon\\
&\equiv \Big[ z(-\Delta_U-z)^{-1}+1\Big]\Upsilon:\ H^s(\partial U)\to H^{s+\frac{1}{2}}_\Delta(U)
\end{split}
\end{align}
will be called the Poisson operator for $U$ and $z$. By construction, it is a bounded operator. Moreover, the adjoint operator of $P_z^U: L^2(\partial U)\to L^2(U)$ is given by:
\[
P^\ast_z=-\gamma^{\partial U}_N(-\Delta_U-\Bar{z})^{-1}: L^2(U)\to L^2(\partial U).
\]
\end{lemma}

\begin{proof} One can uniquely represent $f:=g+\Upsilon\varphi$ with $g\in\dom(-\Delta_U)$. The substitution into the problem shows that
$f$ is a solution if and only if $(-\Delta-z)g=-(-\Delta-z)\Upsilon\varphi$.

Due to the properties of $\Upsilon$ this is equivalent to
$(-\Delta_U-z)g=z\Upsilon\varphi$. As $\Upsilon\varphi\in L^2(U)$ and $z\in \CC\setminus \spec (-\Delta_U)$, one has a unique  solution $g$ given by $g=z(-\Delta_U-z)^{-1}\Upsilon\varphi$, which proves the first statement of the lemma.

Let $\varphi\in L^2(\partial U)$ and $v\in L^2(U)$. As $(-\Delta_U-\Bar{z})^{-1}v\in \dom (-\Delta_U)$ for any $z\in \CC\setminus \spec (-\Delta_U)$, the Green's formula yields
\begin{align*}
    \langle P_z^U \varphi,v\rangle_{L^2(U)}&=\langle P_z^U \varphi,(-\Delta-\Bar{z})(-\Delta-\Bar{z})^{-1}v\rangle_{L^2(U)}\\
   & = -\langle \gamma_D^{\partial U} P_z^U \varphi,\gamma_N^{\partial U}(-\Delta-\Bar{z})^{-1}v\rangle_{L^2(\partial U)}\\
   & = -\langle \varphi,\gamma_N^{\partial U}(-\Delta-\Bar{z})^{-1}v\rangle_{L^2(\partial U)},
\end{align*}
and this completes the proof.
\end{proof}

\begin{defin}For a fixed $s\in[0,1]$ and  $z\in \CC\setminus \spec (-\Delta_U)$ we define the Dirichlet-to-Neumann map $N^U_z$ by
\[
N^U_z:=\gamma^{\partial U}_N P^U_z:\ H^s(\partial U)\to H^{s-1}(\partial U).
\]
By construction, it is again a bounded operator.
\end{defin}

\begin{prop}\label{prop17} Let $s\in[0,1]$ be fixed. For each $z\in \CC\setminus \spec (-\Delta_U)$ the operator
\[
N^U_z-N^U_0:H^s(\partial U)\to L^{2}(\partial U)
\]
is well-defined and compact.
\end{prop}

\begin{proof} Using the representation \eqref{PO} we have
	\[
	N^U_z-N^U_0=\gamma^{\partial U}_N(P^U_z-P^U_0)=z\gamma^{\partial U}_N (-\Delta_U-z)^{-1}\Upsilon.
\]
The embedding $J:H^{s+\frac{1}{2}}(U)\to L^2(U)$ is compact for any $s\in[0,1]$, and one can rewrite
\[
N^U_z-N^U_0=z\gamma^{\partial U}_N (-\Delta_U-z)^{-1}J\Upsilon,
\]	
and the compactness follows from the boundedness of all the other factors.	
\end{proof}
	
The following relation will be important:
\begin{prop}\label{prop18a}
	For any $\psi\in H^{\frac{1}{2}}(\partial U)$ there holds
	\begin{equation}
		 \label{ident1}
	S_{\partial U} N^U_0\psi=\Big(\frac{1}{2}-K_{\partial U}\Big)\psi.
	\end{equation}
\end{prop}

\begin{proof}
For $g\in L^2(\partial U)$ and $x\in U$ define the single and double layer potentials,
	\begin{align*}
		\cS g (x)&=\int_{\partial U} \Phi(x-y)g(y)\dd s(y),\\
		\cD g (x)&=\int_{\partial U} \frac{\langle \nu(y),y-x\rangle}{\sigma_n|x-y|^n} g(y)\dd s(y).
	\end{align*}
It is known that $\cS:H^{-\frac{1}{2}}(\partial U)\to H^1_\Delta(U)$  and $\cD:L^2(\partial U)\to H^1_\Delta(U)$ are well-defined and bounded,
see \cite[Theorem 1]{cost2}. In addition, we have the following identities:
\begin{equation}
	\label{Jump1}
\gamma^{\partial U}_D \cS=S_{\partial U},
\quad
\gamma^{\partial U}_N \cS=-\frac{1}{2}+K'_{\partial U},
\quad
\gamma^{\partial U}_D \cD=\frac{1}{2}+K_{\partial U},
\end{equation}
see~\cite[Theorem 6.11 and Eq. (7.5)]{Mc}.

Let $\psi\in H^{\frac{1}{2}}(\partial U)$ and set $u=P^{U}_0 \psi\in H^1_\Delta(U)$. By \cite[Theorem 5.6.1]{medkova} we have 
\begin{align*}
	u= \cD \gamma^{\partial U}_D u +\cS \gamma^{\partial U}_N u\equiv \cD \psi +\cS N^U_0 \psi.
\end{align*}
Using  \eqref{Jump1} we get 
\[
\psi=\gamma^{\partial U}_D \cD \psi+\gamma^{\partial U}_D \cS N^U_0 \psi=
\Big(\frac{1}{2}+K_{\partial U}\Big)\psi+  \cS N^U_0 \psi. \qedhere
\]
\end{proof}

The following relation will be important:
\begin{equation}
	\label{skk}
	S_{\partial U}K'_{\partial U}\psi=K_{\partial U} S_{\partial U}\psi, \quad\forall \psi\in H^{-\frac{1}{2}}(\partial U).
\end{equation}
known as the symmetrization formula \cite[Proposition 5.7.1]{medkova}.

For subsequent use, we mention the following identity, which follows from the integration by parts:
\begin{lemma}\label{lem19}
	For any $\varphi\in H^1(\partial U)$, any $z\in \CC\setminus\spec(-\Delta_U)$ and $f:=P^U_z \varphi$ there holds
	\[
	\langle\varphi,N^U_z \varphi\rangle_{L^2(\partial U)}=\int_U|\nabla f|^2\dd x-z\int_U|f|^2\dd x.
	\]
	In particular, $N^U_z$ is a symmetric operator in $L^2(\partial U)$ for real $z$.
\end{lemma}

We discuss in greater detail the case when $\partial U$ is disconnected. Let $\Lambda$ be a connected component of $\partial U$ and $\Lambda_0:=\partial U\setminus \Gamma\ne \emptyset$. This gives rise to decompositions 
\[
H^s(\partial U)\simeq H^s(\Lambda)\oplus H^s(\Lambda_0)
\]
and similar decompositions for other functional spaces over $U$, as well as to the representation of
trace maps $\gamma^{\partial U}_{D/N}$ as
\[
\gamma^{\partial U}_{D/N} f=(\gamma^{\Lambda}_{D/N} f, \gamma^{\Lambda_0}_{D/N} f),
\quad
\gamma^{\Lambda}_{D/N} f:=\gamma^{\partial U}_{D/N} f|_{\Lambda},
\quad
\gamma^{\Lambda_0}_{D/N} f:=\gamma^{\partial U}_{D/N} f|_{\Lambda_0}.
\]
For $s\in[0,1]$ and $z\notin \spec(-\Delta_U)$ we introduce the partial Dirichlet-to-Neumann operators
\[
\Tilde N^U_z:\ 
H^s(\Lambda)\ni \varphi\mapsto \gamma^{\Lambda}_N P^U_z(\varphi,0)\in H^{s-1}(\Lambda),
\]
i.e. $\Tilde N^U_z\varphi:=\gamma^\Lambda_N f$, where $f\in H^{s+\frac{1}{2}}_\Delta(U)$ is the unique solution to
\[
(-\Delta-z)f=0 \text{ in } U,\quad
\gamma^\Lambda_D f=\varphi,\quad
\gamma^{\Lambda_0}_D f=0.
\]
In view of the relation with $N^U_z$ one obtains
\begin{prop}\label{prop18}
	For each $s\in[0,1]$ and $z\in \CC\setminus \spec (-\Delta_U)$ the operator $\Tilde N^U_z-\Tilde N^U_0:H^{s}(\Lambda)\to L^2(\Lambda)$ is compact.
\end{prop}

We will need the following variation of Proposition \ref{prop18a}:

\begin{prop}\label{prop111} For any $s\in[\frac{1}{2},1]$ and $\psi\in H^s(\Lambda)$ one has
\[
S_{\Lambda} \Tilde N^U_0\psi=\Big(\frac{1}{2}-K_{\Lambda}\Big)\psi +B\psi,
\]
where $B:H^s(\Lambda)\to H^s(\Lambda)$ is a compact operator.
\end{prop}

\begin{proof}
The operators $K_{\partial U}$ and $S_{\partial U}$ are decomposed as
\[
K_{\partial U}=\begin{pmatrix}
	K_\Lambda & K_{12}\\
	K_{21} & K_{\Lambda_0}
\end{pmatrix}
\quad
S_{\partial U}=\begin{pmatrix}
	S_\Lambda & S_{12}\\
	S_{21} & S_{\Lambda_0}
\end{pmatrix},
\quad
N^{\partial U}_0=\begin{pmatrix}
	\Tilde N^U_0 & N_{12}\\
	N_{21} & N_{22}
\end{pmatrix},
\]
so taking the upper left corner in  \eqref{ident1} gives
\[
S_{\Lambda} \Tilde N^U_0 +S_{12} N_{21}=\frac{1}{2}-K_{\Lambda},
\]
where by construction
\[
S_{12}:H^{-s}(\Lambda_0)\to H^{1-s}(\Lambda),\qquad
N_{21}:H^s(\Lambda)\to H^{s-1}(\Lambda_0)
\]
are bounded operators. The operator $S_{12}: L^2(\Lambda_0)\to H^1(\Lambda)$
has a Lipschitz integral kernel
\[
\Lambda\times\Lambda_0\ni (x,y)\mapsto \Phi(x-y)
\]
(as the diagonal singularity is excluded), so it defines a compact operator from $L^2(\Lambda_0)$ to $H^1(\Lambda)$ by the subsequent Lemma~\ref{lemcomp}. By duality, also the operator $S_{12}:H^{-1}(\Lambda_0)\to L^2(\Lambda)$ is compact, and then the real interpolation theorem for compact operators, see e.g.  \cite[Theorem 2.1]{CF},
implies that $S_{12}$ defines a compact operator $H^{1-s}(\Lambda_0)\to H^s(\Lambda)$. Therefore, $S_{12} N_{21}:H^s(\Lambda)\to H^s(\Lambda)$ is compact.
\end{proof}

\begin{lemma}\label{lemcomp}
Let $\Lambda,\Lambda_0\subset\RR^n$ be two compact  disjoint Lipschitz hypersurfaces and $F:\Lambda\times\Lambda_0\to \CC$
be a Lipschitz function. Then the operator
\[
S:L^2(\Lambda_0)\to H^1(\Lambda),\qquad
S g(x):=\int_{\Lambda_0} F(x,y) g(y)\dd s(y),
\]
is well-defined and compact.
\end{lemma}
\begin{proof}
Let $\big\{(U_j,V_j,\psi_j)\big\}_{j\in\{1,\dots,N\}}$, $N\in\NN$, be an atlas on $\Lambda$ consisting of local charts \[
\psi_j:\RR^{n-1}\supset U_j\to V_j\subset \Lambda.
\]
Let $\chi_j\in C^{\infty}_c(V_j)$ form a subordinated partition of unity ($\chi_1+\dots+\chi_N=1$)
on $\Lambda$. Then an equivalent norm on $H^1(\Lambda)$ is given by
\[
H^1(\Lambda)\ni g\mapsto \sum_{j=1}^N \big\|(\chi_j g)\circ \psi_j\big\|_{H^1_0(U_j)}.
\]
Representing
\begin{align*}
(S g)\big(\psi_j(u)\big)&=\sum_{j=1}^N \int_{\Lambda_0} F_j(u,y) g(y)\dd s(y),\\
F_j(u,y)&:=\chi_j\big(\psi_j(u)\big)F\big(\psi_j(u),y\big),
\end{align*}
we conclude that it is sufficient to show that each of the operators $S_j$,
\[
S_j:\ L^2(\Lambda_0)\ni g\mapsto \int_{\Lambda_0} F_j(u,y) g(y)\dd s(y)\in L^2(U_j)
\]
defines a compact operator from $L^2(\Lambda)$ to $H^1_0(U_j)$.

Let $j\in\{1,\dots,N\}$ be fixed. As $F_j$ is Lipschitz by construction, for each $k\in\{1,\dots,n-1\}$ the function
$\partial_{u_k} F_j$ is bounded. In particular, the integral operators
\[
G_{j,k}:\ L^2(\Lambda_0)\ni g\mapsto \int_{\Lambda_0} \partial_{u_k} F_j(u,y) g(y)\dd s(y)\in L^2(U_j)
\]
are Hilbert-Schmidt, therefore, compact. The dominated convergence shows the identity $\partial_k S_j=G_{j,k}$.

Now consider the operator
\[
M_j: H^{1}_0(U_j)\ni f\mapsto \nabla f \in L^2(U_j,\CC^{n-1}),
\]
where $\nabla$ is the Euclidian gradient. The operator $M_j$ is an isomorphism between $H^1_0(U_j)$
and the closed subspace $\ran M_j$ in $L^2(U_j,\CC^{n-1})$, so the inverse
\[
M_j^{-1}:\ran M_j\to H^1_0(U_j)
\]
is bounded. We now consider the operator
\[
G_j:L^2(\Lambda_0)\ni g\mapsto \begin{pmatrix}
	G_{j,1}\\ \vdots \\ G_{j,n-1}
	\end{pmatrix}
	\equiv
	\begin{pmatrix}
		\partial_1 S_j\\ \vdots \\ \partial_{n-1}S_j
	\end{pmatrix}
	\in \ran M_j,
\]
which is compact by the above considerations, then
\[
S_j=M_j^{-1} G_j:\ L^2(\Lambda)\to H^1_0(U_j)
\]
is a compact operator being a composition of a compact operator with a bounded operator.
\end{proof}

\section{Transmission problem for Laplacians}\label{sec3}

Let us consider the geometric configuration described in the introduction, see Fig.~\ref{fig1}.
Denote by $\nu$ the unit normal vector on $\Sigma$, exterior with respect to $\Omega_-$,
and by $\nu_0$ the exterior unit normal on $\partial\Omega$. Consider the Dirichlet traces
\[
\gamma^\pm_D:\ H^1(\Omega_\pm)\to L^2(\Sigma), \quad \gamma^\partial_D:\ H^1(\Omega_+)\to L^2(\partial\Omega)
\]
defined by
\begin{align*}
    \gamma^\pm_D f:=(f_\pm)|_{\Sigma}\quad\text{and}\quad \gamma^\partial_D f:=(f_+)|_{\partial\Omega} \quad\text{ for } f\in H^1(\Omega).
\end{align*}
Similarly, we let 
\begin{align*}
\gamma^\pm_N:\ H^{s}_{\Delta}(\Omega_\pm)\to H^{s-\frac{3}{2}}(\Sigma) \text{ for each } s\in [1,\tfrac{3}{2}],
\end{align*}
be the Neumann traces defined first by 
\begin{align*}
    \gamma^\pm_N f:= \mp\langle \nu, (\nabla f_\pm)|_{\Sigma} \rangle  \quad\text{ for } f_\pm\in C^{\infty}(\overline{\Omega_\pm}).
\end{align*}
and extended by continuity for any $f_\pm\in H^{\frac{3}{2}}_{\Delta}(\Omega_\pm)$.

In addition, let $s\in[1,\frac{3}{2}]$ and $\mu\in \RR\setminus\{0,1\}$. Consider the following linear operator $A_{(s)}$ in $L^2(\Omega)$:
\begin{align*}
\dom A_{(s)}=\big\{&
f=(f_+,f_-)\in H^s_\Delta(\Omega_+)\oplus H^s_\Delta(\Omega_-)\simeq H^s_\Delta(\Omega\setminus\Sigma): \\
&\gamma_D^{\partial}f_+=0 \text{ on } \partial\Omega,\ 
\gamma_D^-f_-=\gamma_D^+f_+ \text{ on } \Sigma,\\
&\mu\gamma_{N}^-f_-+\gamma_{N}^+f_+=0 \text{ on } \Sigma
\big\},\\
A_{(s)}:&\quad (f_+,f_-)\mapsto (-\Delta f_+,-\mu\Delta f_-).
\end{align*}
one has obviously $C^\infty_c(\Omega\setminus\Sigma)\subset \dom A_{(s)}$, so $A_{(s)}$ is densely defined.
A direct application of the Green formula (Proposition \ref{prop12}) shows that $A_{(s)}$ is a symmetric operator in $L^2(\Omega)$.
We are going to find several conditions guaranteeing its self-adjoitnness. Remark that for $\mu=1$ the operator $A_{(\frac{3}{2})}$ is well-defined as well, but coincides with the Dirichlet Laplacian in $\Omega$, so we exclude this trivial case from the very beginning. We emphasize on the fact that no assumption on the sign of $\mu$ is made.

Let $-\Delta_\pm$ be the Dirichlet Laplacian in $L^2(\Omega_\pm)$ and denote
\[
B:=(-\Delta_+)\oplus(-\mu\Delta_-),
\]
which is a self-adjoint operator with compact resolvent in $L^2(\Omega)$.
Remark that both $A_{(s)}$ and $B$ are restrictions of the following linear operator $T$ in $L^2(\Omega)$:
\begin{align*}
	\dom T_{(s)}&=\Big\{
	f=(f_+,f_-)\in H^s_\Delta(\Omega_+)\oplus H^s_\Delta(\Omega_-)\simeq H^s_\Delta(\Omega\setminus\Sigma): \\
	&\qquad \gamma_D^{\partial}f_+=0 \text{ on } \partial\Omega,\ 
	\gamma_D^-f_-=\gamma_D^+f_+ \text{ on } \Sigma
	\Big\},\\
	T_{(s)}:&\quad (f_+,f_-)\mapsto (-\Delta f_+,-\mu\Delta f_-).
\end{align*}
Consider the symmetric linear operator $\Theta$ in $L^2(\Sigma)$,
\[
\dom \Theta=H^{s-\frac{1}{2}}(\Sigma), \quad \Theta =\Tilde N^+_0+\mu N^-_0,
\quad
\Tilde N^+_0: \ H^{s-\frac{1}{2}}(\Sigma)\ni \varphi\mapsto N^+_0(\varphi,0)|_\Sigma.
\]
and the linear maps $\Gamma_0^{(s)},\Gamma_1^{(s)}: \dom T_{(s)}\to H^{s-\frac{3}{2}}(\Sigma)$ given by
\[
\Gamma_0^{(s)} f:=\gamma_D^-f_-\equiv \gamma_D^+f_+,
\quad
\Gamma_1^{(s)} f:=-(\mu\gamma_{N}^-f_-+\gamma_{N}^+f_+)+\Theta_{(s)} \Gamma_0^{(s)} f.
\]
Due to the results of Section \ref{secint} we have
\[
\ran\Gamma_0^{(s)}=H^{s-\frac{1}{2}}(\Sigma),\quad \ran \Gamma_1^{(s)}\subset H^{s-\frac{3}{2}}(\Sigma).
\]
For any $f,g\in \dom T_{(s)}$ one has, using the integration by parts,
\begin{align*}
\langle T_{(s)}f,g\rangle_{L^2(\Omega)}-\langle f,T_{(s)}g\rangle_{L^2(\Omega)}&
=\langle \Gamma_1^{(s)} f,\Gamma_0^{(s)} g\rangle_{H^{s-\frac{3}{2}}(\Sigma),H^{\frac{3}{2}-s}(\Sigma)}\\
&\quad- \langle \Gamma_0^{(s)} f,\Gamma_1^{(s)} g\rangle_{H^{s-\frac{3}{2}}(\Sigma),H^{\frac{3}{2}-s}(\Sigma)}.
\end{align*}
In addition, the domains of $A_{(s)}$ and $B$ can be represented as
\begin{align*}
\dom A_{(s)}&=\{f\in \dom T_{(s)}:\ \Gamma_0^{(s)} f\in\dom \Theta_{(s)},\ \Gamma_1^{(s)} f=\Theta \Gamma_0^{(s)} f\},
\\ \dom B&=\ker \Gamma_0^{(1)}.
\end{align*}
Remark that the above properties of $T_{(\frac{3}{2})},\Gamma_0^{(\frac{3}{2})}$ and $\Gamma_1^{(\frac{3}{2})}$ ensure that the triple $\big( L^2(\Sigma),\Gamma_0^{(\frac{3}{2})},\Gamma_1^{(\frac{3}{2})}\big)$ is a so-called
quasi boundary triple for the adjoint $S^*$ of the symmetric operator
\[
S:=T_{\ker\Gamma_0^{(\frac{3}{2})}\cap\ker\Gamma_1^{(\frac{3}{2})}},
\]
see \cite[Theorem 2.3]{bl07}; we also refer to \cite{bl12,dhm1,dhm2} for a review of various classes of boundary triples.
On the other hand, even for this particular case the assumptions for most useful conclusions are extremely difficult to check, so we need to adapt manually the central ideas of the boundary triple approach to our specific situation.
The main result is a formula relating the resolvents of $A_{(s)}$ and $B$, which we are going to state.

Let $z\in \CC\setminus\spec(B)$, then we have a direct sum decomposition
\[
\dom T_{(s)}=\dom B \dot + \ker (T_{(s)}-z),
\]
which shows that $\Gamma_0^{(s)}:\ker(T_{(s)}-z)\to \ran\Gamma_0^{(s)}$ is invertible, and we denote the inverse by $P_z$. By construction, for any $\varphi\in H^{s-\frac{1}{2}}(\Sigma)$ the function $f\equiv (f_+,f_-):=P_z\varphi$ is the unique solution to
\begin{gather*}
(-\Delta-z)f_+=0 \text{ in } \Omega_+,\qquad (-\mu \Delta-z)f_-=0 \text{ in } \Omega_-,\\
\gamma^\partial_D f_+=0, \qquad \gamma^+_D f_+=\varphi=\gamma^-_D f_-.
\end{gather*}
Due to the results of Section \ref{secint} the solution takes the form
\[
f_+:=P^+_z(\varphi,0),\quad f_-:=P^-_{\frac{z}{\mu}}\varphi,
\]
where $P^\pm_z$ are the Poisson operators for $\Omega_\pm$ and $z$ and one uses the natural identifications
\[
L^2(\partial\Omega_+)\simeq L^2(\Sigma)\oplus L^2(\partial\Omega),
\quad
H^1(\partial\Omega_+)\simeq H^1(\Sigma)\oplus H^1(\partial\Omega) \quad \text{etc.}
\]
In particular,
\[
P_z:\ H^{s-\frac{1}{2}}(\Sigma)\to H^s_\Delta(\Omega\setminus\Sigma),\quad s\in[1,\tfrac{3}{2}].
\]
is bounded.

We will also be interested in the operator (Weyl function)
\[
M_z:=\Gamma_1 P_z:H^{s-\frac{1}{2}}(\Sigma)\to H^{s-\frac{3}{2}}(\Sigma),
\]
which takes the form
\begin{gather*}
M_z=\Theta-(\Tilde N^+_z+\mu N^-_{\frac{z}{\mu}})\equiv (\Tilde N^+_0-\Tilde N^+_z)+\mu (N^-_0-N^-_{\frac{z}{\mu}}),\\
\quad\Tilde N^+_z:  H^{s-\frac{1}{2}}(\Sigma)\ni \varphi\mapsto N^+_z(\varphi,0)|_\Sigma,
\end{gather*}
and is a compact operator from $H^{s-\frac{1}{2}}(\Sigma)$ to $L^2(\Sigma)$ by Proposition \ref{prop17}.

 The following Theorem gives us some key tools for proving the self-adjointness of $A_{(s)}$. We recall that for the case $s=\frac{3}{2}$ the results follow from abstract results on quasi boundary triples \cite[Theorem 2.8]{bl07}.

\begin{theorem}\label{propaa}Let $s\in[1,\frac{3}{2}]$ be fixed, then the following hold:
\begin{itemize}
    \item[(a)] For any  $z\in \CC\setminus\spec(B)$ one has the equality
		\[
		\ker(A_{(s)}-z)=P_z \Big(\big\{ \psi\in H^{s-\frac{1}{2}}(\Sigma): g\in(\ker(\Theta-M_z) \big\}\Big).
		\]
		In particular,  $\Theta-M_z:H^{s-\frac{1}{2}}(\Sigma)\to H^{s-\frac{3}{2}}(\Sigma)$ is injective for all non-real $z$, as $A_{(s)}$ is symmetric.
		
		\item[(b)] Let $z\in \CC\setminus\spec(B)$ and $\Theta-M_z:H^{s-\frac{1}{2}}(\Sigma)\to H^{s-\frac{3}{2}}(\Sigma)$ is injective, and let $P^*_{\Bar z}:L^2(\Omega)\to L^2(\Sigma)$ be the adjoint of $P_{\Bar z}:L^2(\Sigma)\to L^2(\Omega)$.
		
		If $f\in L^2(\Omega)$ satisfies $P^*_{\Bar{z}}f\in (\Theta-M_z)\big(H^{s-\frac{1}{2}}(\Sigma)\big)$, then there holds $f\in\ran (A_{(s)}-z)$ and 
		\begin{align}\label{resfor}
			(A_{(s)}-z)^{-1}f= (B-z)^{-1}f +P_z (\Theta-M_z)^{-1}P^*_{\Bar{z}}f.
		\end{align}
  \end{itemize}
\end{theorem}
\begin{proof} Fix $s\in[1,\frac{3}{2}]$, let $z\in \CC\setminus\spec(B)$ and $\mu\neq0$.

     (i) By definition of $\Theta-M_z$ it is enough to show that $u\in\mathit{H}^{s}_{\Delta}(\Omega\setminus\Sigma)\setminus\{0\}$ satisfies
	\begin{equation}\label{Sys}
		\left\{
		\begin{aligned}
			(-\Delta-z)u_+&=0, \quad \text{ in } \Omega_+,\\
			(-\mu\Delta-z)u_-&=0, \quad \text{ in } \Omega_-,\\
			\gamma^{\partial}_D u_+&= 0, \quad \text{ on } \partial\Omega,
		\end{aligned}
		\right.
	\end{equation}
	if and only if there is $\psi\in \mathit{H}^{s-\frac{1}{2}}(\Sigma)$ such that $u=P_z\psi$.
 
 So let  $\psi\in H^{s-\frac{1}{2}}(\Sigma)\setminus\{ 0\}$. Then it is clear that the function $u=P_z\psi$ belongs to $H^{s}_{\Delta}(\Omega\setminus\Sigma)$ and satisfies the system \eqref{Sys}. Conversely, let $u\in H^{s}_{\Delta}(\Omega\setminus\Sigma)\setminus\{0\}$ be a solution of \eqref{Sys} and set $\psi=\gamma_D^- u_-$. Note that $\psi\neq0$ as otherwise $z$ would belong to the spectrum of $B$. Since $\gamma _D^+P^+_z(\psi,0)=\gamma _D^-P^-_{\frac{\mu}{z}}\psi= \psi=\gamma_D^- u_-$, Lemma \ref{lemma5} implies that $u=P_z\psi$ and this completes the proof of (i).

 (ii) Assume that $\Theta-M_z:H^{s-\frac{1}{2}}(\Sigma)\to H^{s-\frac{3}{2}}(\Sigma)$ is injective.
   Recall that by Lemma \ref{lemma5} there holds 
 \[
 P^*_{\Bar{z}}= -\Big(\gamma^{+}_N(-\Delta_+-z)^{-1}+\mu \gamma_{N}^{_-}(-\mu\Delta_{-}-z)^{-1}\Big) .
 \]

  Let $f\in L^2(\Omega)$ with $P^*_{\Bar{z}}f\in\ran (\Theta-M_z)$, then  $(\Theta-M_z)^{-1}P^*_{\Bar{z}}f$ is well defined and belongs to $H^{s-\frac{1}{2}}(\Sigma)$. Define the function 
 \[
 g= (B-z)^{-1}f +P_z (\Theta-M_z)^{-1}P^*_{\Bar{z}}f.
 \]

 From Lemma \ref{lemma5} one has $g\in H^s(\Omega\setminus\Sigma)$. Moreover, we have
\begin{align*}
\gamma_{N}^{_-}g_-&=  \gamma_{N}^{_-}(-\mu\Delta_{-}-z)^{-1}f_- + N^-_{-\frac{z}{\mu}}(\Theta-M_z)^{-1}P^*_{\Bar{z}}f,\\
\gamma_{N}^{+}g_+&=  \gamma_{N}^{+}(-\Delta_{+}-z)^{-1}f_+ +\Tilde{N}^+_{z}(\Theta-M_z)^{-1}P^*_{\Bar{z}}f.
\end{align*}
Thus, by definition of $\Theta$ and $M_z$  we obtain
\begin{align*}
    \gamma_{N}^{+}g_+  + \mu\gamma_{N}^{-}g_-&= P^*_{\Bar{z}}f - (N^-_{z}+\mu N^+_{-\frac{z}{\mu}})(\Theta-M_z)^{-1}P^*_{\Bar{z}}f\\
    &\equiv P^*_{\Bar{z}}f - (\Theta-M_z)(\Theta-M_z)^{-1}P^*_{\Bar{z}}f =0.
\end{align*}
Similarly, one checks $\gamma_{D}^+g_+=\gamma_{D}^-g_-$ and $\gamma_{D}^{\partial\Omega}g_+=0$,
which means that $g\in\dom A_{(s)}$.

From the proof of (i) it follows that $(A_{(s)}-z)g=f$, so $f\in\ran (A_{(s)}-z)$. This completes the proof of~\eqref{resfor}.
\end{proof}

\begin{corol}\label{cormain}
	Let $s\in[1,\frac{3}{2}]$ and suppose that
	\[
	L^2(\Sigma)\subset(\Theta-M_z)\big(H^{s-\frac{1}{2}}(\Sigma)\big)
	\]
	for some $z\in\CC\setminus\RR$, then $A_{(s)}$ is self-adjoint.
\end{corol}

\begin{proof}
The assumption ensures
\[
\ran P^*_{\Bar z}\subset L^2(\Sigma)\subset(\Theta-M_z)\big(H^{s-\frac{1}{2}}(\Sigma)\big),
\]
and by Theorem \ref{propaa} one obtains $\ran(A_{(s)}-z)=L^2(\Omega)$.	As $A_{(s)}$ is symmetric and commutes with the comlpex conjugation, this shows the self-adjointness.	
\end{proof}

\begin{corol}\label{cormain2}
Let $s\in[1,\frac{3}{2}]$ and suppose that
\[
(\Theta-M_z)\big(H^{s-\frac{1}{2}}(\Sigma)\big)=H^{s-\frac{3}{2}}(\Sigma)
\]
for some $z\in\CC\setminus\RR$, then $A_{(s)}$ is self-adjoint.
\end{corol}

\begin{proof}
Follows by Corollary \ref{cormain} due to the inclusion $L^2(\Sigma)\subset H^{s-\frac{3}{2}}(\Sigma)$.
\end{proof}

\section{Sufficient conditions for the self-adjointness}

Our main tool to verify the assumption of Corollary \ref{cormain} is:
 
\begin{theorem}\label{thm22} Let $s\in[ \frac{1}{2},1]$ and $\mu\in\RR\setminus\{0,1\}$ be such that:
	\begin{itemize}
		\item[(a)] $K'_\Sigma-\dfrac{\mu+1}{2(\mu-1)}: H^{s-1}(\Sigma)\to H^{s-1}(\Sigma)$ is Fredholm of index $m\in\ZZ$,
		\item[(b)] $S_\Sigma:H^{s-1}(\Sigma)\to H^{s}(\Sigma)$ is Fredholm with index zero.
	\end{itemize}
Then the operator $\Theta-M_z: H^{s-\frac{1}{2}}(\Sigma)\to H^{s-\frac{3}{2}}(\Sigma)$ is Fredholm with the same index $m$ for any non-real $z$.
\end{theorem}

\begin{proof} 
Fix $s\in[\frac{1}{2},1]$ and let $z\in\CC\setminus\RR$. Remark first that
\[
\Theta-M_z\equiv\Tilde N^+_z+\mu N^-_\frac{z}{\mu}\equiv \Tilde N^+_0+\mu N^-_0+B_1,
\]	
where $B_1:H^{s}(\Sigma)\to L^2(\Sigma)$ is a compact operator, see Propositions \ref{prop17} and \ref{prop18}, and
\begin{equation}
	\label{mtheta1}
\big(\Theta-M_z\big)S_\Sigma=\Tilde N^+_0 S_\Sigma+\mu N^-_0 S_\Sigma+B_2,
\end{equation}
where $B_2:H^{s-1}(\Sigma)\to L^2(\Sigma)$ is a compact operator. 

The Fredholmness of $S$ implies the existence of ``an approximate inverse'' $T:H^{s}(\Sigma)\to H^{s-1}(\Sigma)$ such that 
\begin{align}\label{fredT}
    S_\Sigma T= 1 + R_1 \quad\text{and}\quad TS_\Sigma = 1 + R_2 
\end{align}
where $R_1: H^{s}(\Sigma)\to H^s(\Sigma)$ and $R_2: H^{s-1}(\Sigma)\to H^{s-1}(\Sigma)$ are compact operators,
and $T:H^s(\Sigma)\to H^{s-1}(\Sigma)$ is automatically Fredholm with index zero.
By Propositions \ref{prop18a} and \ref{prop111} it follows that
\begin{align*}
    N^-_0=\,T\Big(\dfrac{1}{2}-K_\Sigma\Big) + B_3\quad\text{and}\quad
    \Tilde N^+_0=T\Big(\frac{1}{2}+K_\Sigma\Big) +B_4
\end{align*}
where $B_3,\, B_4:H^s(\Sigma)\to H^{s-1}(\Sigma)$ are compact operators (the minus sign in front of $K_\Sigma$ is due to the fact that the normal $\nu$ is interior with respect to $\Omega_+$). The substitution of these expressions into \eqref{mtheta1} gives
\[
\big(\Theta-M_z\big)S_\Sigma=T \Big( \frac{1}{2}+K_\Sigma + \frac{\mu}{2}-\mu K_\Sigma\Big)S_\Sigma+B_5,
\]
where $B_5:=B_2+(B_3+B_4) S_\Sigma:H^{s-1}(\Sigma)\to H^{s-1}(\Sigma)$ is a compact operator. This can be rewritten as
\[
(\Theta-M_z)S_\Sigma=(1-\mu)T \Big( K_\Sigma-\dfrac{\mu+1}{2(\mu-1)}\Big)S_\Sigma+B_5.
\]
Using the symmetrization formula \eqref{skk} and \eqref{fredT} we rewrite the above as
\begin{equation}
	\label{temp111}
\begin{aligned}
(\Theta-M_z)S_\Sigma&=(1-\mu)\Big( K'_\Sigma-\dfrac{\mu+1}{2(\mu-1)}\Big)+R_2\Big( K'_\Sigma-\dfrac{\mu+1}{2(\mu-1)}\Big)+B_5\\
    &\equiv(1-\mu)\Big( K'_\Sigma-\dfrac{\mu+1}{2(\mu-1)}\Big)+B_6,
\end{aligned}
\end{equation}
where $B_6$ is compact in $H^{s-1}(\Sigma)$.

Let $K'_\Sigma-\frac{\mu+1}{2(\mu-1)}$ be Fredholm with index $m$ in $H^{s-1}(\Sigma)$. Due to the compactness of $B_6$,
\eqref{temp111} shows that $(\Theta-M_z)S_\Sigma:H^{s-1}(\Sigma)\to H^{s-1}(\Sigma)$ is Fredholm with index $m$, and then
$(\Theta-M_z)S_\Sigma T: H^{s}(\Sigma)\to H^{s-1}(\Sigma)$ is Fredholm with index $m$ as well. By \eqref{fredT} we have
\begin{align*}
\Theta-M_z= (\Theta-M_z)S_\Sigma T - (\Theta-M_z)R_1
\end{align*}
and $(\Theta-M_z)R_1: H^s(\Sigma)\to H^{s-1}(\Sigma)$ is compact, which gives the required conclusion.
\end{proof}

For what follows, for a bounded linear operator $T$ in a Banach space $\mathcal{B}$ it will be convenient to denote
\begin{align*}
\spec_\ess(T, \mathcal{B})&:=\big\{ z\in\CC: \ T-z \text{ is not a Fredholm operator}\big\},\\
\spec^0_\ess(C,\mathcal{B})&:=\big\{ z\in\CC: \ T-z \text{ is not a zero index Fredholm operator}\big\}.
\end{align*}
On has the trivial inclusions
\[
\spec_\ess(T,\mathcal{B})\subset \spec^0_\ess(T,\mathcal{B})\subset\spec(T,\mathcal{B})
\]
and the equalities
\begin{align*}
\spec(T^*,\mathcal{B}^{\ast})&=\overline{\spec(T,\mathcal{B})},\\
\spec_\ess(T^*,\mathcal{B}^{\ast})&=\overline{\spec_\ess(T,\mathcal{B})},\\
\spec^0_\ess(T^*,\mathcal{B}^{\ast})&=\overline{\spec^0_\ess(T,\mathcal{B})}.
\end{align*}
In addition define the spectral radius $r(T,\mathcal{B})$ and the essential spectral radius $r_\ess(T,\mathcal{B})$ by
\begin{align*}
r(T,\mathcal{B})&:=\sup\big\{|\lambda|:\,\lambda\in\spec(T,\mathcal{B})\big\},\\
r_\ess(T,\mathcal{B})&:=\sup\big\{|\lambda|:\,\lambda\in\spec_\ess(T,\mathcal{B})\big\}.
\end{align*}
One has clearly
\begin{align*}
r_\ess(T,\mathcal{B})&\le r(T,\mathcal{B}),\\
r(T^*,\mathcal{B}^{\ast})&=r(T,\mathcal{B}),\\
r_\ess(T^*,\mathcal{B}^{\ast})&=r_\ess(T,\mathcal{B}).
\end{align*}
 Recall that the index of $T-\lambda$ is constant for $\lambda$ in each connected component of $\CC\setminus\spec_\ess(T,\mathcal{B})$. As the index is zero for $|\lambda|>\|T\|$, we conclude that
\[
\sigma_\ess(T,\mathcal{B})\subset\big\{\lambda\in \CC:\, |\lambda|\le r_\ess(T,\mathcal{B})\big\}.
\]

With these conventions and observations, Theorem \ref{thm22} has the following direct consequence.
\begin{theorem}\label{corself}
Let $s\in [1,\frac{3}{2}]$ and $\mu\in\RR\setminus\{0, 1\}$ such that:
\begin{itemize}
	\item[(a)] $\frac{\mu+1}{2(\mu-1)}\notin\spec^0_\ess(K'_\Sigma, H^{s-\frac{3}{2}}(\Sigma))$,
	\item[(b)] $S_\Sigma:H^{s-\frac{3}{2}}(\Sigma)\to H^{s-\frac{1}{2}}(\Sigma)$ is Fredholm with index zero.
\end{itemize}
Then $A_{(s)}$ is self-adjoint with compact resolvent.
\end{theorem}

\begin{proof}
For any non-real $z$ the operator $\Theta-M_z:H^{s-\frac12}(\Sigma)\to H^{s-\frac32}(\Sigma)$ is Fredholm with index zero by Theorem \ref{thm22} and injective by Theorem \ref{propaa}(a). Therefore, $(\Theta-M_z)\big(H^{s-\frac12}(\Sigma)\big)=H^{s-\frac32}(\Sigma)$ for any non-real $z$,
and the self-adjointness of $A_{(s)}$ follows by Corollary \ref{cormain2}. 
	
It remains to show that $(A_{(s)}-z)^{-1}$ is compact in $L^2(\Omega)$ for at least one non-real $z$.
Let $z\in\CC\setminus\RR$. Since $\Theta-M_z:H^{s-\frac{1}{2}}(\Sigma)\to H^{s-\frac{3}{2}}(\Sigma)$ is bounded and bijective (as just shown above), the closed graph theorem shows that $(\Theta-M_z)^{-1}:H^{s-\frac{3}{2}}(\Sigma)\to H^{s-\frac{1}{2}}(\Sigma)$ is bounded.
Due to the compact embedding $H^{s-\frac{1}{2}}(\Sigma)\hookrightarrow L^2(\Sigma)$
the operator $(\Theta-M_z)^{-1}:H^{s-\frac{3}{2}}(\Sigma)\to L^2(\Sigma)$ is compact. Since
$(B-z)^{-1}:L^2(\Omega)\to L^2(\Omega)$ is compact and
\[P_z:L^2(\Sigma)\to L^2(\Omega),\quad
P^*_{\Bar z}:L^2(\Omega)\to L^2(\Sigma)
\]
are bounded, the resolvent formula in Theorem \ref{propaa} shows the compactness of $(A_{(s)}-z)^{-1}$.
\end{proof}

We remark that the second assumption (b) is much less restrictive than the other conditions (in fact, we believe that it is always satisfied, but we were unable to find a sufficiently general result for $n=2$ in the available literature).
\begin{prop}\label{propss}
	The assumption (b) in Theorem \ref{corself} is satisfied in the following cases:
	\begin{itemize}
		\item[(i)] for all $s\in[1,\tfrac32]$, if $n\ge 3$,
		\item[(i)] at least for  $s\in\{1,\tfrac32\}$, if $n=2$.
	\end{itemize}
\end{prop}

\begin{proof}
	For (i) we refer to Theorem 8.1 (5) in~\cite{fabes}.
	For (ii), see Theorem 4.11 in \cite{verchota} for $s=\frac32$ and Theorem 7.6 in \cite{Mc} for $s=1$.
\end{proof}

By combining Theorem \ref{corself} with Proposition \ref{propss} we arrive at the following
main observation:

\begin{corol}\label{corself2}
Let $s\in \{1,\frac{3}{2}\}$ and $\mu\in\RR\setminus\{0, 1\}$ such that
\begin{equation}
	\label{aaa}
	\frac{\mu+1}{2(\mu-1)}\notin\spec^0_\ess\big(K'_\Sigma, H^{s-\frac{3}{2}}(\Sigma)\big),
\end{equation}	
then $A_{(s)}$ is self-adjoint with compact resolvent.
The condition \eqref{aaa} is satisfied, in particular,
	if
	\[
	\Big|\frac{\mu+1}{2(\mu-1)}\Big|> r_\ess\Big(K'_\Sigma,, H^{s-\frac{3}{2}}(\Sigma)\Big).
	\]
\end{corol}

We now establish several concrete situations in which the assumption of Corollary \ref{corself2}
is satisfied.  First remark that
\begin{equation}
	\label{rks}
	 r\Big(K'_\Sigma, L^2(\Sigma)\Big)\le\frac{1}{2}\quad\text{ and }\quad r\Big(K'_\Sigma, H^{-\frac{1}{2}}(\Sigma)\Big)\le\frac{1}{2},
\end{equation}	
see \cite[Theorem 2]{EFV} and \cite[Theorem 2]{CL}.

\begin{theorem}[{\bf Sign-definite case}]\label{thm17}
For any $\mu>0$ the operator $A_{(\frac{3}{2})}$ is self-adjoint  (without any additional assumption on $\Sigma$)
and, in addition,  $A_{(\frac{3}{2})}=A_{(1)}$.
\end{theorem}

\begin{proof}
Let $\mu>0$ with $\mu\ne 1$, then due to \eqref{rks} we have
\[
\Big|\frac{\mu+1}{2(\mu-1)}\Big|>\frac{1}{2},
\]
and (a) in Theorem \ref{corself} is satisfied for both $s=1$ and $s=\frac{3}{2}$ due to \eqref{rks}.
In view of Proposition \ref{propss} this shows that both $A_{(1)}$ and $A_{(\frac{3}{2})}$ are self-adjoint.
Due to $A_{(\frac{3}{2})}\subset A_{(1)}$ the two operators must coincide (as self-adjoint operators are always maximal).
The remaining case $\mu=1$ is trivial.
\end{proof}

Recall \cite{HMT} that $\mathrm{VMO}(\Sigma)$ stands for the space of functions of vanishing means oscillation on $\Sigma$. The bounded Lipschitz domains with normals in $\mathrm{VMO}$ are those domains ''without corners''.
This includes all bounded $C^1$-smooth domains.

\begin{theorem}[{\bf Domains without corners}]\label{thm18}
Let $\Sigma$ be such that $\nu\in \mathrm{VMO}(\Sigma)$, which is satisfied, in particular for $C^1$-smooth $\Sigma$.
Then the operator $A_{(\frac{3}{2})}$ is self-adjoint for any $\mu\in\RR\setminus\{-1,0\}$.
\end{theorem}	

\begin{proof}
The operator $K'_{\Sigma}$ is compact on $L^2(\Sigma)$ by \cite[Theorem 4.47]{HMT}, so $\spec^0_\ess\big(K'_\Sigma, L^2(\Sigma)\big)=\{0\}$ and $r_\ess\big(K'_\Sigma, L^2(\Sigma)\big)=0$, and for $\mu\notin\{-1,0,1\}$ the claim follows by Corollary \ref{corself2}. The case $\mu=1$ is trivially covered.
\end{proof}

\begin{remark}[{\bf Bounds for spectral radii}]\label{rmk19}
So far we have no self-adjointness condition for the interesting case $\mu<0$ and $\Sigma$ with corners. 
In fact, if for a given $\Sigma$ one can improve \eqref{rks} to 
\[
r(s):=r_\ess\big(K'_\Sigma, H^{s-\frac{3}{2}}(\Sigma)\big)<\frac{1}{2},
\]
for some $s\in[1,\frac{3}{2}]$, then for $\mu\notin\{0,1\}$ there holds
\[
\Big|\frac{\mu+1}{2(\mu-1)}\Big|>r(s)
\text{ if and only if }
\mu\notin I_{r(s)}:=\Bigg[-\frac{1+2r(s)}{1-2r(s)}, -\frac{1-2r(s)}{1+2r(s)}\Bigg],
\]
and the key assumption (a) in Theorem \ref{corself} is satisfied for all $\mu$ outside the ``critical interval'' $I_{r(s)}$ (in particular, a range of negative values is admitted).
At the same time, the inequalities
\[
r\big(K'_\Sigma, L^2(\Sigma)\big)<\frac{1}{2},\quad r_\ess\big(K'_\Sigma, L^2(\Sigma)\big)<\frac{1}{2}
\]
represent central conjectures in the theory of Neumann-Poincar\'e operators, which are still unsolved in the general form. The papers \cite{hag,wendland} contain a rather complete bibliography and a review of available results on the spectral radii. Let us discuss an important particular case in greater detail.
\end{remark}

\begin{theorem}\label{thm20}
Let $n=2$ and $\Sigma$ be a curvilinear polygon with $C^1$-smooth edges and  with $N$ interior angles $\omega_1,\dots,\omega_N\in(0,2\pi)\setminus\{0\}$.
Let $\omega\in(0,\pi)$ be the sharpest corner, i.e.
\begin{align*}
    \dfrac{|\pi-\omega|}{2}&=\max_k \dfrac{|\pi-\omega_k|}{2},
\end{align*}
then the following hold:
\begin{itemize}
    \item[(i)] The operator $A_{(\frac{3}{2})}$ is self-adjoint with compact resolvent for all $\mu\ne0$ such that
\begin{equation*}
\mu\notin I_{r(\frac{3}{2})}=\Big[-\dfrac{1}{a(\omega)},-a(\omega)\Big]
\text{ for } a(\omega):=\tan^2\frac{\omega}{4}\equiv\dfrac{1-\cos\frac{\omega}{2}}{1+\cos\frac{\omega}{2}}.
\end{equation*}
\item[(ii)] If in addition the edges are $C^2$-smooth, then the operator $A_{(1)}$ is self-adjoint with compact resolvent for all $\mu\ne0$ such that
\begin{equation*}
\mu\notin I_{r(1)}=\big[-\dfrac{1}{b(\omega)}, -b(\omega)\big]
\text{ for }
b(\omega):=\dfrac{\pi-|\pi-\omega|}{\pi +|\pi-\omega|}.
\end{equation*}
\end{itemize}
\end{theorem}

\begin{proof} (i) By	\cite[Theorem 6]{VYS} there holds
\[
	r_\ess\big(K'_\Sigma, L^2(\Sigma)\big)=\dfrac{1}{2}\max_k \sin\dfrac{|\pi-\omega_k|}{2}\equiv \dfrac{1}{2}\sin\dfrac{|\pi-\omega|}{2},
\]
and the rest follows from the computations in Remark \ref{rmk19}
with the help of elementary trigonometric identities. Similarly, from \cite[Theorem 7]{PP} one has 
\[
	r_\ess\big(K'_\Sigma, H^{-\frac{1}{2}}(\Sigma)\big)= \frac{1}{2\pi }\max_k| \pi-\omega_k|\equiv \frac{1}{2\pi}|\pi-\omega|,
\]
and thus assertion (ii) is a direct consequence of Remark \ref{rmk19}.
\end{proof}

\begin{remark}\label{rmk21}
The case of curvilinear polygons was already considered (with a slightly requirements on the regularity of the edges)
in \cite{bbdr,bccj,DT} for $s=1$. A numerical comparison of the bounds $a(\omega)$ and $b(\omega)$ is shown in Figure \ref{fig2}.
\end{remark}

\begin{figure}
\centering
\includegraphics[height=30mm]{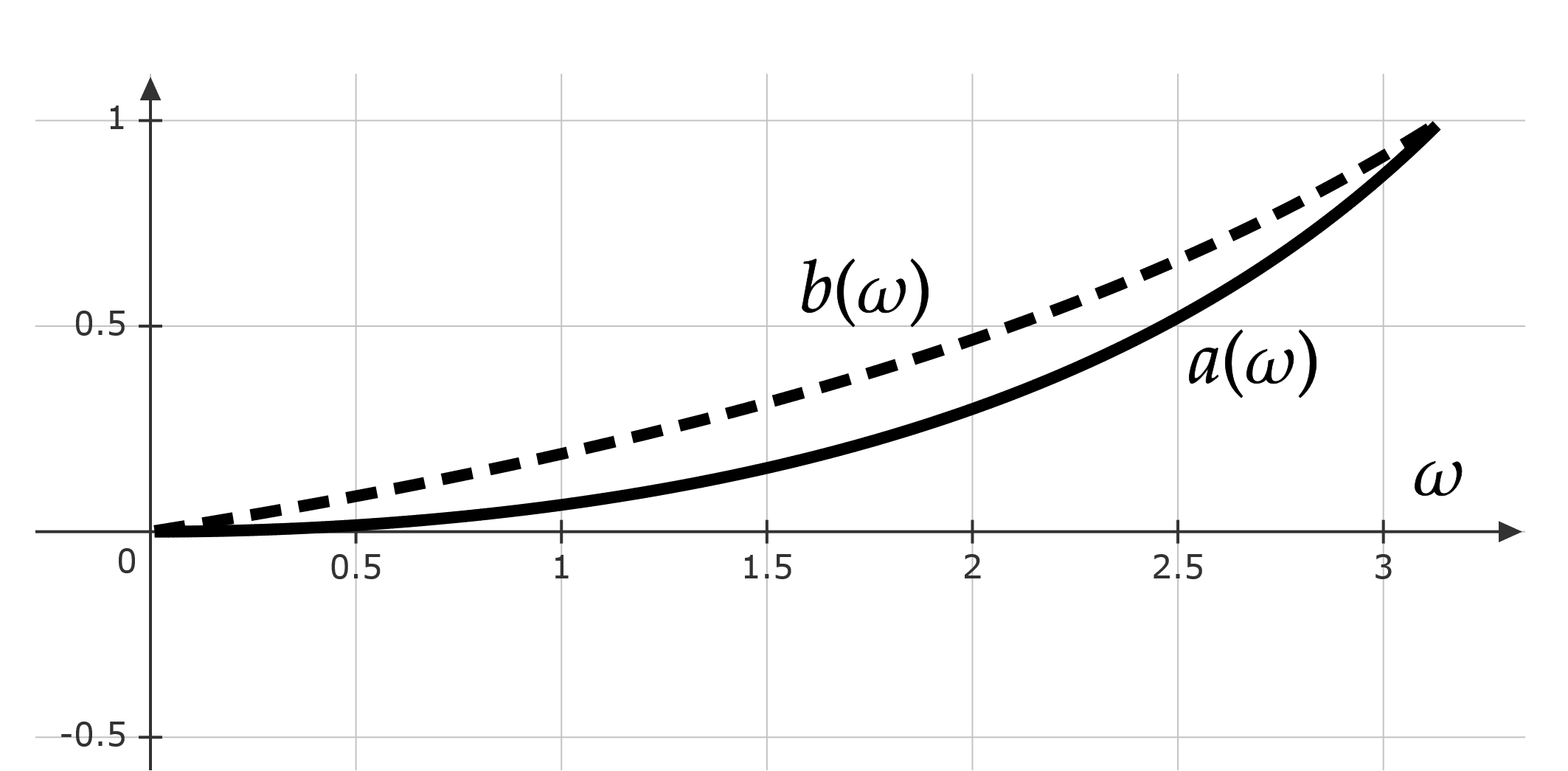}
	\caption{Plots of $\omega\mapsto a(\omega)$ (solid line) and $\omega\mapsto b(\omega)$ (dashed line).}\label{fig2}
\end{figure}

\begin{remark}
We are not aware of efficient estimates for $r_\ess\big(K'_\Sigma, L^2(\Sigma)\big)$ and $r_\ess\big(K'_\Sigma, H^{-\frac{1}{2}}(\Sigma)\big)$ if $n\ge 3$ and $\Sigma$ has corners. In \cite[Section 4]{kmp1} it was shown that if
$\Sigma$ is an arbitrary three-dimensional polyhedron, then $r_\ess(K'_\Sigma, L^2(\Sigma))<\frac{1}{2}$, but no constructive upper bound was available. In \cite[Theorem 5.5]{km2} one finds an expression for the $r_\ess(K'_\Sigma, L^2(\Sigma))$ for the case of rotationally invariant conical singularities in three dimensions in terms of modified Legendre functions, but the expressions obtained did not contain an explicit bound of $r_\ess(K'_\Sigma, L^2(\Sigma))$. Nevertheless,
some useful information on the conical points can be extracted as follows.	
\end{remark}

\begin{theorem}[{\bf Conical points}]\label{thmcon} Let $n=3$ and assume that $\Sigma\subset\RR^3$ is obtained by revolution of a $C^5$-smooth
	curve $\gamma$ and that $\Sigma$ is $C^1$-smooth except for one conical singularity that forms an angle $\alpha\in(0,\pi)\setminus\{\frac{\pi}{2}\}$ with the rotational axis. Then, the following hold:
	\begin{itemize}
		\item If $0<\alpha<\frac{\pi}{2}$ then $A_{(1)}$ is self-adjoint with compact resolvent for any $\mu>-1$.
		\item If $\frac{\pi}{2}<\alpha<\pi$ then $A_{(1)}$ is self-adjoint with compact resolvent for any $\mu<-1$.
	\end{itemize}
\end{theorem}
\begin{proof} From \cite[Lemma 5]{LPS} one has 
	\[
	\spec^0_\ess\big(K'_\Sigma, H^{-\frac{1}{2}}(\Sigma)\big)=\begin{cases} [0,C_\alpha], & 0<\alpha<\frac{\pi}{2},\\
		[-C'_\alpha,0], & \frac{\pi}{2}<\alpha<\pi
	\end{cases}
	\]
	with suitable constant $C_\alpha,C'_\alpha>0$ expressed in terms of special functions. Then the result is a direct consequence of Corollary \ref{corself2}. 
\end{proof}

\begin{remark} For each $s\in[1,\frac{3}{2}]$ we believe that a more general statement should be valid: if $K'_{\Sigma}-\frac{\mu+1}{2(\mu-1)}$ is Fredholm of index $m$ in $H^{s-\frac{3}{2}}(\Sigma)$, then $A_{(s)}$ is a closed symmetric operator with deficiency indices $(m,m)$. This should follow by a suitable extension of the resolvent formula, which would in turn require some non-evident and adapted extensions of the Dirichlet/Neumann traces. We prefer to leave this question for future work.    
\end{remark}

\section*{Acknowledgments}
The authors were in part supported by the Deutsche Forschungsgemeinschaft (DFG, German Research Foundation) -- 491606144.
The authors are thankful to Irina Mitrea and Daniel Grieser for useful bibliographical hints.

\end{document}